\documentclass{elsarticle}
\usepackage[usenames,dvipsnames]{xcolor}

\usepackage{./TLC}

\usepackage{IEEEtrantools}
\usepackage{mathtools}
\usepackage{physics}
\usepackage{siunitx}

\begin{document}
\title{A rolled-off passivity theorem}

\author[1]{Thomas Chaffey}
\ead{tlc37@cam.ac.uk}
\affiliation[1]{organization={University of Cambridge, Department of
Engineering},addressline={Trumpington Street},postcode={CB2~1PZ},city={Cambridge},country={United Kingdom}}

\begin{abstract}
        Given two nonlinear systems which only violate incremental
        passivity when their incremental gains are sufficiently small, 
        we give a condition for their negative feedback interconnection to 
        have finite incremental gain, which generalizes the incremental small gain and incremental
        passivity theorems.  The property may be determined graphically by plotting the
        Scaled Relative Graphs (SRGs) of the systems, which provides engineering significance to the
        mathematical result.
\end{abstract}


\maketitle

\section{Introduction}

The small gain and passivity theorems are two of the fundamental pillars of
nonlinear input/output systems theory, originating in a seminal paper of Zames
\cite{Zames1966}.  The former theorem states that the
negative feedback of two systems is stable if the product of their gains is less than
one; the latter guarantees stability if one system is passive and the other is
strongly passive with finite gain.

Zames provides two versions of each theorem.  The first version requires gain and
phase properties to be satisfied with respect to the reference input/output pair
$u_0(t) = 0, y_0(t) = 0$, and guarantees boundedness,
or continuity of the operator at $u_0$.  The second version requires properties to be
satisfied \emph{incrementally}, or with respect to every input/output trajectory, and
guarantees the stronger property of continuity everywhere.  These stronger theorems
are known as the \emph{incremental small gain} and \emph{incremental passivity}
theorems.

The conditions of the small gain and passivity theorems are restrictive, and there
are many feedback systems which are stable, but do not meet the assumptions of either
theorem.  
A common issue in practice is that a system would satisfy the conditions of the
passivity theorem, were it not for high frequency dynamics
destroying passivity. The input/output gain, however, is small at these high
frequencies.  This issue played an important part in the development of adaptive
control - see \cite{Anderson2005} and references therein.  The prevalence of such
systems has motivated several specialized stability results. The LTI \emph{mixed
small gain/passivity} condition of Griggs et al.\@ \cite{Griggs2007a} divides the
frequency spectrum into frequencies at which two systems are passive, frequencies at
which they have small gain, and frequencies at which they satisfy both criteria.
This is generalized directly to nonlinear systems by Forbes and Damaren
\cite{Forbes2010}, using the terminology \emph{hybrid small gain/passivity}, and is connected to the Generalized KYP lemma of Iwasaki and Hara \cite{Iwasaki2005} in reference \cite{Forbes2013}.  The
nonlinear generalization of Griggs et al.\@ \cite{Griggs2009} uses a pair of linear operators to define
a ``blended'' supply rate which represents mixed small gain and passivity.
The \emph{roll-off IQC} was introduced by Summers, Arcak and Packard
\cite{Summers2013} to capture the roll-off of input/output gain at high frequency.  
All these mixed small gain/passivity results are
non-incremental, guaranteeing boundedness of the feedback system, but not continuity.
The secant condition also mixes small gain and passivity, guaranteeing the stability of a cascade of output-strictly passive
systems with a condition on the product of their \emph{secant gains}, which capture
both gain and passivity information \cite{Sontag2006, Arcak2006}.  The secant
condition readily applies to the incremental case \cite{Stan2007}.  Recent work has
also generalized the passivity theorem to infinite dimensional LTI systems
\cite{Zhao2017, Guiver2017}.

The subject of this paper is an incremental \emph{rolled-off passivity} theorem,
which applies to incrementally stable systems which only violate incremental passivity when their
incremental input/output gain is small.  Rather than requiring gain to
roll-off over any particular frequency range, we simply require the gain to roll off
when the phase shift, measured as an angle in signal space, exceeds $\pi/2$.
The idea takes inspiration from the blended supply rate of Griggs et al.\@ \cite{Griggs2009},
however, rather than using a smoothly blended supply, we simply split the 
space of input signals into those pairs of signals where the two systems have incremental small gain,
are incrementally passive, or both.  We call this property \emph{incremental $(\mu,
\gamma)$-dissipativity}.  Unlike the results of references
\cite{Forbes2010, Griggs2009, Griggs2007}, we do not require systems to be
incrementally $(\mu, \gamma)$-dissipative for \emph{the same}
partition of signals. This maintains the ``worst case'' nature of the small gain and
passivity theorems, but simplifies the verification of the property. 
The resulting condition for finite
incremental gain bears a strong resemblance to the classical incremental small gain
condition.

Methods exist for determining when LTI systems satisfy a mixed or hybrid small gain/passivity
property \cite{Forbes2013, Griggs2011}.  Determining whether such properties hold for
a pair of nonlinear systems, however, is often difficult.  In contrast, incremental $(\mu,
\gamma)$-dissipativity can be read directly from the Scaled Relative Graph (SRG) of a
nonlinear system.
SRGs have recently been introduced by Ryu, Hannah and Yin \cite{Ryu2021} for the
study of monotone operator methods in optimization.  The SRG allows incremental
properties of operators, such as Lipschitz continuity and monotonicity, to be
determined graphically, and leads to intuitive and rigorous proofs of convergence
using geometric transformations in the complex plane.  The SRG is particularly
suited to proving tightness of bounds, and novel tightness results have been obtained
\cite{Ryu2021, Huang2020a}.  In reference \cite{Chaffey2021c}, the author and his
colleagues showed that the SRG generalizes the Nyquist diagram of an LTI transfer
function (generalized by Pates
\cite{Pates2021}), and applied SRG techniques to the study of feedback systems.
Theorem 2 of reference \cite{Chaffey2021c} shows that if SRGs corresponding to two
systems are separated (and remain separated as one SRG is scaled), the negative feedback of the two systems has finite incremental
gain.  The distance between the two SRGs is a nonlinear stability margin. This result generalizes the Nyquist criterion \cite{Nyquist1932} to stable nonlinear operators.  The
rolled-off passivity condition that we give in this paper guarantees that the
relevant SRGs are separated, and is a special case of \cite[Thm. 2]{Chaffey2021c}. In
order to keep this paper self-contained, we give a direct proof of the
result.

There is, of course, a large body of work on more general stability theorems for
feedback systems.
For systems of the Lur'e form, that is, an LTI dynamic component in negative
feedback with a static nonlinearity,  dynamic multipliers, such as those proposed by
Popov \cite{Popov1964} and Zames, Falb and O'Shea \cite{Zames1968, OShea1966,
OShea1967}, can be used to make the LTI
component passive, without affecting the passivity of the static nonlinearity.  The
passivity theorem can then be used to conclude stability.  Methods involving the gap and 
related metrics prove stability by showing the input/output graphs of two systems are
separated \cite{El-Sakkary1985a, Vinnicombe2000, Georgiou1990a, Foias1993, Georgiou1997}. 
Absolute stability and multiplier methods are 
unified in the IQC framework introduced by Megretski and Rantzer
\cite{Megretski1997}.  These methods guarantee boundedness of the input/output
operator, but in general give no guarantee of continuity.  Indeed, it has been shown
that large classes of multipliers, while preserving passivity of static
nonlinearities, destroy their incremental passivity, making it difficult to conclude
continuity from multiplier methods \cite{Kulkarni2001}.  Continuity at a particular
point in signal space, which isn't necessarily $u_0(t) = 0, y_0(t) = 0$, can be guaranteed by the use of differential 
techniques \cite{Manchester2018, Wang2019}.  The primary benefits of the result we
present here are the guarantee of continuity everywhere, the ability of incremental $(\mu, \gamma)$-dissipativity to capture 
common real world effects, and the ability to verify the required
properties graphically.

After introducing some necessary preliminaries and giving a brief review of the theory of SRGs in Section~\ref{sec:prelims}, 
we formally introduce the property of incremental $(\mu, \gamma)$-dissipativity in Section~\ref{sec:rolled-off}, and give it a graphical interpretation.  We then state the main result of this
paper in Section~\ref{sec:theorem}, Theorem~\ref{thm:small_passivity}, and give a direct proof of the result.  An
example is given in Section~\ref{sec:example}.

\section{Preliminaries}
\label{sec:prelims}

\subsection{Signals and systems}
We model systems as operators on a Hilbert space, which is a vector space of signals,
equipped with an inner product $\bra{\cdot}\ket{\cdot}$ and induced norm $\norm{\cdot} \coloneqq \sqrt{\bra{\cdot}\ket{\cdot}}$.
Let $\mathbb{F} \in \{\R, \C\}$, and let $L_2(\mathbb{F}^n)$ denote the Hilbert space of
signals $u: \R_{\geq 0} \to \mathbb{F}^n$ such that
\begin{IEEEeqnarray*}{rCl}
        \norm{u} \coloneqq \left(\int^\infty_0 \bar u(t) u(t) \dd{t}\right)^{\frac{1}{2}} < \infty,
\end{IEEEeqnarray*}
where $\bar{u}(t)$ denotes the conjugate transpose of $u(t)$.
The inner product on $L_2(\mathbb{F}^n)$ is given by
\begin{IEEEeqnarray*}{rCl}
        \bra{u}\ket{y} \coloneqq \int^\infty_0 \bar u(t)y(t) \dd{t}.
\end{IEEEeqnarray*}
We write $L_2$ for $L_2(\R^n)$, where the dimension $n$ is immaterial.

By \emph{an operator} (on a Hilbert space $\mathcal{H}$), we mean a possibly multi-valued map $H:
\mathcal{H} \to \mathcal{H}$.  The
\emph{graph}, or \emph{relation}, of $H$, is the set $\{(u, y) \,|\, u \in \mathcal{H}, y
\in H(u)\}$.  We will use the notions of an operator and its relation
interchangeably, and denote them the same way.  The usual operations on functions can
be extended to relations:
\begin{IEEEeqnarray*}{rCl}
        H^{-1} &=& \{ (y, u) \; | \; y \in H(u) \}\\
        H + G &=& \{ (x, y + z) \; | \; (x, y) \in H, (x, z) \in G \}\\
        HG &=& \{ (x, z) \; | \; \exists\; y \text{ s.t. } (x, y) \in G, (y, z) \in H \}.
\end{IEEEeqnarray*}
Note that $H^{-1}$ always exists, but is not an inverse in the usual sense.  In
particular, it is in general not the case that $H^{-1}H = I$.  If, however, $H$ is an
invertible function, its functional inverse coincides with its relational inverse, so
the notation $H^{-1}$ can be used without ambiguity.

An operator $H: \mathcal{H} \to \mathcal{H}$ is said to have \emph{finite incremental
gain}, or be \emph{Lipschitz continuous}, if
there exists some nonnegative $\mu < \infty$ such that, for all $u_1, u_2 \in
\mathcal{H}$,
$y_1 \in H(u_1)$ and $y_2 \in H(u_2)$,
\begin{IEEEeqnarray*}{rCl}
        \norm{y_1 - y_2} &\leq \mu \norm{u_1- u_2}.
\end{IEEEeqnarray*}

We say that an operator $H: \mathcal{H} \to \mathcal{H}$ is \emph{incrementally positive} if, for all $u_1, u_2 \in
\mathcal{H}$,
$y_1 \in H(u_1)$ and $y_2 \in H(u_2)$,
\begin{IEEEeqnarray*}{rCl}
        \bra{u_1 - u_2}\ket{y_1 - y_2} \geq 0.
\end{IEEEeqnarray*}
We say that $H$ is \emph{$\epsilon$-input strictly incrementally positive} if, for all $u_1, u_2 \in
\mathcal{H}$,
$y_1 \in H(u_1)$ and $y_2 \in H(u_2)$,
\begin{IEEEeqnarray*}{rCl}
        \bra{u_1 - u_2}\ket{y_1 - y_2} \geq \epsilon\norm{u_1 - u_2}^2.
\end{IEEEeqnarray*}
Incremental positivity is, in general, a weaker property than incremental passivity
(as defined by Zames \cite{Zames1966}), however, 
the two are equivalent for causal operators on $L_2$ \cite[p. 174]{Desoer1975}.
Incremental positivity is otherwise known as \emph{(operator) monotonicity}, as introduced by
Minty in 1961 \cite{Minty1961}, and popularized by Rockafellar
\cite{Rockafellar1976}.  Monotonicity has since become a fundamental property in the
field of mathematical optimization \cite{Bauschke2011}.  We use the term
\emph{incremental positivity}, adopted by Zames \cite{Zames1966}, Desoer and
Vidyasagar \cite{Desoer1975}, and others, partly because the results of this paper
form a natural extension of their work, and partly to avoid confusion with the
unrelated notion of monotonicity introduced by Hirsch and Smith \cite{Hirsch2006}.

\subsection{Scaled relative graphs}\label{sec:srg}

In this section, we briefly introduce the theory of SRGs.  We give only the theory
required for the proof of Theorem~\ref{thm:small_passivity}, and refer the interested
reader to reference \cite{Ryu2021} for the complete theory of SRG interconnections and their
use in the theory of optimization, reference
\cite{Chaffey2021c} for the use of SRGs in systems theory and their relation to the
Nyquist diagram, and references \cite{Ryu2021, Huang2020a, Chaffey2021c, Pates2021} for the computation of SRGs for particular systems.

\subsubsection{Definition}

We define the SRG on a general, possibly complex, Hilbert space $\mathcal{H}$.

The angle between $u, y \in \mathcal{H}$ is defined as
\begin{IEEEeqnarray*}{rCl}
        \angle(u, y) \coloneqq \acos \frac{\Re \bra{u}\ket{y}}{\norm{u}\norm{y}}. 
\end{IEEEeqnarray*}
This angle is in $[0, \pi]$.

Let $H: \mathcal{H} \to \mathcal{H}$ be an operator.  Given $u_1, u_2 \in
\mathcal{H}$, $u_1 \neq u_2$, define the set of complex numbers $z_H(u_1, u_2)$ by
\begin{IEEEeqnarray*}{rCl}
        z_H(u_1, u_2) \coloneqq &&\left\{\frac{\norm{y_1 - y_2}}{\norm{u_1 - u_2}} e^{\pm j\angle(u_1 -
u_2, y_1 - y_2)}\middle|\; y_1 \in H(u_1), y_2 \in H(u_2) \right\}.
\end{IEEEeqnarray*}
If $u_1 = u_2$ and there exist corresponding
outputs $y_1 \in H(u_1), y_2 \in H(u_2), y_1 \neq y_2$, then
$z_H(u_1, u_2)$ is defined to be $\{\infty\}$.  If $H$ is single valued at $u_1$,
$z_H(u_1, u_1)$ is the empty set.

The \emph{Scaled Relative Graph} (SRG) of $H$ over $\mathcal{H}$ is then given by
\begin{IEEEeqnarray*}{rCl}
        \srg{H} \coloneqq \bigcup_{u_1, u_2 \in\, \mathcal{H}}  z_H(u_1, u_2).
\end{IEEEeqnarray*}

The SRG of an operator is a region in the extended complex plane, 
symmetric about the real axis, from which properties of the
operator can be easily read.  Each point on the SRG gives the relative gain and phase shift of the operator
for one or more particular pairs of inputs.   The SRG can be thought of as a nonlinear
generalization of the Nyquist diagram.  This is elaborated on in the following
example.
\begin{example} \label{ex:LTI}
        The upper half of the SRG of a stable LTI transfer function $G$ is the \emph{hyperbolic-convex hull}
        of the upper half of its Nyquist diagram \cite[Thm. 4]{Chaffey2021c}.  Mathematically, this is
        the set
        \begin{IEEEeqnarray*}{rCl}
        (f\circ g)^{-1}(\operatorname{Co}((f\circ g)(\operatorname{Nyquist}^+(G)))),
        \end{IEEEeqnarray*}
        where $\operatorname{Co}$ is the convex hull, $\operatorname{Nyquist}^+(G)
        \subseteq \C$ is the Nyquist diagram of $G$ restricted to $\{z \in \C :
        \Im(z) \geq 0\}$, and $f, g: \C \to \C$ are the
        mappings
        \begin{IEEEeqnarray*}{rCl}
        f(z) &=& \frac{2z}{1 + |z|^2},\\
        g(z) &=& \frac{z - j}{z + j}.
        \end{IEEEeqnarray*}
        Intuitively, the hyperbolic-convex hull is obtained by taking the convex hull
        with arcs centred on the real axis, rather than straight lines.  For example,
        the SRG of $1/(s + 1)$ is its Nyquist diagram, the circle with centre at
        $0.5$ and radius $0.5$.  The SRG of $e^{-s}/(s+1)$ is illustrated in
        Figure~\ref{fig:delay_srg}. 
\end{example}

\begin{figure}[hb]
        \centering
        \includegraphics{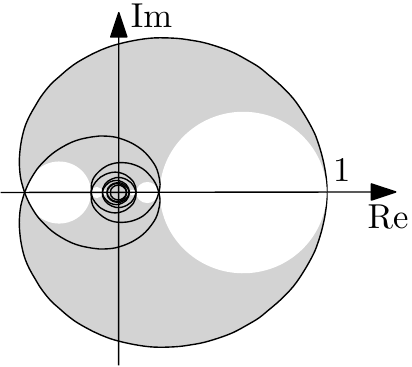}
        \caption{Nyquist diagram (black) and SRG (grey) of the transfer function
        $e^{-s}/(s+1)$.}%
        \label{fig:delay_srg}
\end{figure}

We also give a simple example of the SRG of a nonlinear operator.

\begin{example} \label{ex:sat}
        Define the unit saturation $\operatorname{sat}(\cdot)$ by
\begin{IEEEeqnarray*}{rCl}
        y &=& \begin{cases}
                u & |u| < 1\\
                u/|u|& \text{otherwise}.
              \end{cases}
\end{IEEEeqnarray*}
The SRG of the unit saturation is the closed disc with centre at $0.5$ and radius
$0.5$: $\{z\; \large{|}\; |z - 0.5| \leq 0.5\}$.  This is proved in \cite[Prop.
12]{Chaffey2021c}.
\end{example}

\subsubsection{Determining system properties from the SRG}

The properties which can be read from the SRG are those that define \emph{SRG-full classes}.
If $\mathcal{A}$ is a class of operators, we define the SRG of $\mathcal{A}$ by
\begin{IEEEeqnarray*}{rCl}
        \srg{\mathcal{A}} \coloneqq \bigcup_{H \in \mathcal{A}} \srg{H}.
\end{IEEEeqnarray*}
Note that a class of operators can be a single operator, and the operators in a class
need not act on the same Hilbert space.

A class $\mathcal{A}$, or its SRG, is called \emph{SRG-full} if
\begin{IEEEeqnarray*}{rCl}
        H \in \mathcal{A}\quad \iff \quad \srg{H} \subseteq \srg{\mathcal{A}}.
\end{IEEEeqnarray*}
The value of SRG-fullness is in the implication $\srg{H} \subseteq \srg{\mathcal{A}}
\implies H \in \mathcal{A}$.  This allows class membership to be determined
graphically.

Two examples of SRG-full classes are the class of operators which share a finite incremental
gain bound, and the class of incrementally positive operators.

\begin{proposition}\label{prop:finite_gain}
        Let $\mathcal{H}$ denote an arbitrary Hilbert space. 
        An operator $H: \mathcal{H} \to \mathcal{H}$ obeys
         \begin{IEEEeqnarray*}{rCl}
                 \norm{y_1 - y_2} \leq \mu \norm{u_1 - u_2}
         \end{IEEEeqnarray*}
         for every $u_1, u_2 \in \mathcal{H}$, $y_1 \in H(u_1)$ and $y_2 \in H(u_2)$ if, and
         only if, its SRG belongs to the closed disc of radius $\mu$: 
         \begin{IEEEeqnarray*}{rCl}
         \srg{H} \subseteq \{z \, |\, z \in \C, |z| \leq \mu \}.
         \end{IEEEeqnarray*}

        $H$ obeys
         \begin{IEEEeqnarray*}{rCl}
                 \bra{u_1 - u_2}\ket{y_1 - y_2} \geq 0
         \end{IEEEeqnarray*}
         for every $u_1, u_2 \in \mathcal{H}$, $y_1 \in H(u_1)$ and $y_2 \in H(u_2)$ if, and
         only if, its SRG belongs to the closed right half complex plane: 
         \begin{IEEEeqnarray*}{rCl}
         \srg{H} \subseteq \{z \, |\, z \in \C, \Re(z) \geq 0 \}.
         \end{IEEEeqnarray*}
\end{proposition}

\begin{proof}
        The proof  may be found in
        the proof of \cite[Prop. 3.3]{Ryu2021}, or the proof of Lemma~\ref{lem:srg}
        in Section~\ref{sec:rolled-off}.
\end{proof}

\subsubsection{System interconnection}

The SRGs of interconnected systems can be determined or approximated from the SRGs of
the components.  For a full treatment of SRG interconnections, we refer the reader to
\cite{Ryu2021}, Theorems 4.1-4.5.  Here, we describe what happens to the SRG under
input and output scaling, summation, inversion and composition.

If $C, D \subseteq \C$, we define the
operation $C + D$ to be the Minkowski sum of $C$ and
$D$, that is,
\begin{IEEEeqnarray*}{rCl}
        C + D \coloneqq \{c + d\, | \, c \in C, d \in D\}.
\end{IEEEeqnarray*}
The set $CD$ is defined to be the Minkowski product of $C$ and $D$,
\begin{IEEEeqnarray*}{rCl}
        CD \coloneqq \{cd\, | \, c \in C, d \in D\}.
\end{IEEEeqnarray*}
The set $\alpha C$ is defined by
\begin{IEEEeqnarray*}{rCl}
        \alpha C \coloneqq \{ \alpha c\, |\, c \in C\}.
\end{IEEEeqnarray*}

We define inversion in the complex plane by $re^{j\omega} \mapsto (1/r)e^{j\omega}$.
Throughout this paper, any inversion of a complex number refers to this operation,
which maps points outside the unit circle to the
inside, and vice versa. The points $0$ and $\infty$ are exchanged under inversion.
This operation only differs from the usual complex inversion in that the complex
conjugate is not taken; this is left
out for convenience, as it allows us to work in the upper half complex plane. This is
possible as the SRG is symmetric about the real axis.

Given an operator $A$, the operator 
 $\alpha A$ is defined by $u \mapsto \alpha A(u)$,
        and the operator $A\alpha$ is defined by $u \mapsto A(\alpha u)$.  These
        definitions extend to classes of operators in the natural way.  Given two
        classes of operators $\mathcal{A}$ and $\mathcal{B}$, their sum $\mathcal{A}
        + \mathcal{B}$ is defined to be $\{A + B\; | \; A \in \mathcal{A}, B \in
        \mathcal{B}, \dom(A) = \dom(B)\}$.  The composition $\mathcal{A}\mathcal{B}$
        is defined to be $\{AB \; | \; A \in \mathcal{A}, B \in \mathcal{B},
        \operatorname{range}(B) = \dom(A)\}$.

Define the line segment between $z_1, z_2 \in \C$ as $[z_1, z_2] \coloneqq \{\alpha
                                z_1 + (1 - \alpha) z_2\, |\, \alpha \in [0, 1]\}$.
A class of operators $\mathcal{A}$ is said to
satisfy the \emph{chord property} if $z \in \srg{\mathcal{A}}\setminus\{\infty\}$ implies $[z,
\bar z] \subseteq \srg{\mathcal{A}}$.

Define the \emph{right-hand arc}, $\rarc{z, \bar{z}}$, between $z$ and $\bar{z}$ to be the arc between $z$
and $\bar{z}$ with centre on the origin and real part greater than or equal to
$\Re(z)$. The \emph{left-hand arc}, $\larc{z, \bar{z}}$, is defined the same way, but with real part less
than or equal to $\Re(z)$ (see \cite{Ryu2021} for a more formal definition).
A class of operators $\mathcal{A}$ is said to satisfy the \emph{right hand (resp. left hand) arc property} if,
for all $z \in \srg{\mathcal{A}}$, $\rarc{z, \bar{z}} \in \srg{\mathcal{A}}$  (resp.
$\larc{z,
\bar{z}} \in \srg{\mathcal{A}}$).

We have the following interconnection rules for SRGs, which correspond to \cite[Thm.
4.2-4.5]{Ryu2021}.

\begin{proposition}\label{prop:gain}
        Let $\alpha \in \R, \alpha \neq 0$. If $\mathcal{A}$ is a class of operators,
        \begin{IEEEeqnarray*}{rCl}
                \srg{\alpha\mathcal{A}} &=& \srg{\mathcal{A}\alpha} =
                \alpha\srg{\mathcal{A}}.
        \end{IEEEeqnarray*}
        Furthermore, if $\mathcal{A}$ is SRG-full, then $\alpha\mathcal{A}$ and $\mathcal{A}\alpha$
         are SRG-full.
\end{proposition}

\begin{proposition}\label{prop:inversion}
        If $\mathcal{A}$ is a class of operators, then
        \begin{IEEEeqnarray*}{rCl}
                \srg{\mathcal{A}^{-1}} = (\srg{\mathcal{A}})^{-1}.
        \end{IEEEeqnarray*}
        Furthermore, if $\mathcal{A}$ is SRG-full, then $\mathcal{A}^{-1}$ is
        SRG-full.
\end{proposition}

\begin{proposition}\label{prop:summation}
        Let $\mathcal{A}$ and $\mathcal{B}$ be classes of operators, such that 
        $\infty \nin \srg{\mathcal{A}}$ and $\infty \nin \srg{\mathcal{B}}$. Then:
        \begin{enumerate}
                \item if $\mathcal{A}$ and $\mathcal{B}$ are SRG-full, then
                        $\srg{\mathcal{A} + \mathcal{B}} \supseteq \srg{\mathcal{A}} +
                        \srg{\mathcal{B}}$.
                \item if either $\mathcal{A}$ or $\mathcal{B}$ satisfies the chord
                        property, then 
                        $\srg{\mathcal{A} + \mathcal{B}} \subseteq \srg{\mathcal{A}} +
                        \srg{\mathcal{B}}$.
        \end{enumerate}
\end{proposition}
Infinity can be allowed by setting $\srg{\mathcal{A} + \mathcal{B}} = \{\infty\}$ if
$\srg{\mathcal{A}} = \varnothing$ and $\infty \in \srg{\mathcal{B}}$.

\begin{proposition}\label{prop:composition}
        Let $\mathcal{A}$ and $\mathcal{B}$ be classes of operators, such that
        $\srg{\mathcal{A}}$ and $\srg{\mathcal{B}}$ are nonempty and bounded. Then:
        \begin{enumerate}
                \item if $\mathcal{A}$ and $\mathcal{B}$ are SRG-full, then
                        $\srg{\mathcal{A}\mathcal{B}} \supseteq \srg{\mathcal{A}}
                        \srg{\mathcal{B}}$.
                \item if either $\mathcal{A}$ or $\mathcal{B}$ satisfies an arc 
                        property, then 
                        $\srg{\mathcal{A}\mathcal{B}} \subseteq \srg{\mathcal{A}}
                        \srg{\mathcal{B}}$.
        \end{enumerate}
\end{proposition}
Under additional assumptions, unbounded and empty SRGs can be allowed -- see the discussion following \cite[Thm.
4.5]{Ryu2021}.

We conclude this preliminary section with a simple example of SRG composition.

\begin{example}\label{ex:composition}
        Let $\bar G$ be the transfer function $1/(s+1)$.  Denote the corresponding
        operator on $L_2$ by $u \mapsto Gu$. Consider the composition
        $G(\operatorname{sat}(\cdot))$ - the cascade of $G$ with a unit saturation. 
        The SRGs of both $G$ and
        $\operatorname{sat}(\cdot)$ are contained within the disc $\mathcal{U}
        \coloneqq \{z\; |\; |z - 0.5| \leq 0.5\}$, as shown in Examples~\ref{ex:LTI}
        and~\ref{ex:sat}.  Note that the SRG of $\operatorname{sat}(\cdot)$ satisfies
        the right-hand arc property.

        Let $\mathcal{A}$ denote the set of all operators whose SRGs lie within the
        disc $\mathcal{U}$.  Then $G (\operatorname{sat}(\cdot)) \in
        \mathcal{A}\mathcal{A}$.  It follows from statement~2 of Proposition~\ref{prop:composition}
        that $\srg{G (\operatorname{sat}(\cdot))} \subseteq \mathcal{U}\mathcal{U}$.  As
        shown in \cite[Thm. 1]{Huang2020a}, $\mathcal{U}\mathcal{U}$ is the
        cardioid $\{re^{j\phi} \; | \; r \leq \frac{1}{2}(1 + \cos(\phi))\}$,
        illustrated in Figure~\ref{fig:composition}.
\end{example}

\begin{figure}[ht]
        \centering
                \includegraphics{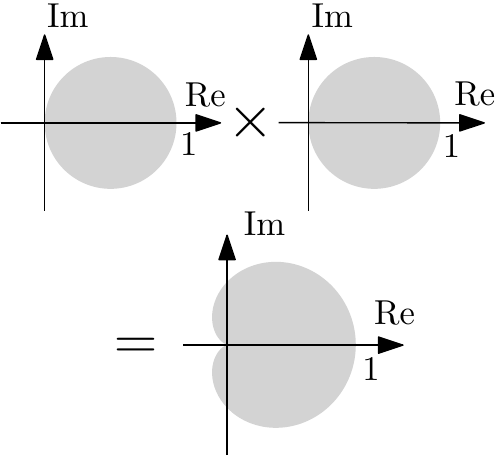}
                \caption{Illustration of the composition used in
                Example~\ref{ex:composition}.  The SRG on the right, a cardioid, is a
        bound on the SRG of a unit saturation cascaded with the transfer function
        $1/(s+1)$.}%
        \label{fig:composition}
\end{figure}

\section{Rolled-off passivity}\label{sec:rolled-off}

In this section, we define a property which captures systems which have a roll-off in gain as
their phase shift increases.  Given a pair of input signals, the system must either be incrementally passive,
with an incremental gain which is finite but in general large, or have small incremental gain.  This is formalized as follows.
\begin{definition}
        Let $H: L_2 \to L_2$ and $\mu, \gamma, \epsilon \geq 0$.  We say
        that $H$ is \emph{$\epsilon$-strongly incrementally $(\mu, \gamma)$-dissipative}
        if, for all $u_1, u_2 \in L_2$ and all $y_1 \in H(u_1)$, $y_2 \in H(u_2)$,
        either:
        \begin{IEEEeqnarray}{l}
                \norm{y_1 - y_2} \leq \mu \norm{u_1- u_2},\label{eq:small_gain}
        \end{IEEEeqnarray}
        or both:
                \begin{subequations}
        \begin{IEEEeqnarray}{*l+rCl}
                \arraycolsep = 1.4pt
                &\bra{u_1 - u_2}\ket{y_1 - y_2} &\geq& \epsilon\norm{u_1 - u_2}^2
                        \label{eq:passivity_one}\\
                 \text{and} &\norm{y_1 - y_2} &\leq& \gamma \norm{u_1- u_2},
                        \label{eq:passivity_two}
        \end{IEEEeqnarray}
                \end{subequations}
                or all of \eqref{eq:small_gain}, \eqref{eq:passivity_one} and \eqref{eq:passivity_two} hold.
        If $\epsilon = 0$, we simply say that $H$ is \emph{incrementally $(\mu, \gamma)$-dissipative}.
\end{definition}

Incremental $(\mu, \gamma)$-dissipativity is defined independently of the frequency spectra of signals, however it captures
those systems which are incrementally passive except for high frequency dynamics, when the 
system has low gain. Incremental $(\mu, \gamma)$-dissipativity is easily verified for systems with low-pass
dynamics whose passivity is destroyed by effects such as input saturation and small delays, as explored
further in the example of Section~\ref{sec:example}.

If $\gamma < \mu$, incremental $(\mu, \gamma)$-dissipativity reduces to an
incremental gain bound of $\mu$.  If $\mu = 0$, the property reduces to finite
incremental gain and input strict incremental positivity.

Incremental $(\mu, \gamma)$-dissipativity has an appealing graphical interpretation, developed in the following lemma.  
This lemma is especially useful as
it allows the property of incremental $(\mu, \gamma)$-dissipativity to be easily determined from the
SRG of a system.

\begin{lemma}\label{lem:srg}
        Let $\mu, \gamma > 0$, $\epsilon \geq 0$, and let $\mathcal{S}_{\mu,
        \gamma}^\epsilon$ be the class of
        operators which are $\epsilon$-strongly incrementally $(\mu,
        \gamma)$-dissipative.  Then
        \begin{IEEEeqnarray*}{rCl}
                \srg{\mathcal{S}_{\mu, \gamma}^\epsilon} = \mathcal{D}_1\cup\mathcal{D}_2,
        \end{IEEEeqnarray*}
        where 
        \begin{IEEEeqnarray*}{rCl}
                \mathcal{D}_1 &\coloneqq& \{z\, | \, z \in \C, |z| \leq \mu\},\\
                \mathcal{D}_2 &\coloneqq& \{z\, | \, z \in \C, |z| \leq
                \gamma, \Re{z} \geq \epsilon\},
        \end{IEEEeqnarray*}
        as shown in Figure~\ref{fig:srg}.
        Furthermore, $\mathcal{S}_{\mu, \gamma}^\epsilon$  is SRG-full.
\end{lemma}
\begin{figure}[ht]
        \centering
        \includegraphics{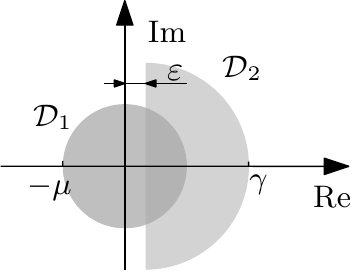}
        \caption{Graphical interpretation of $\epsilon$-strong incremental $(\mu,
                \gamma)$-dissipativity.  Lemma~\ref{lem:srg} shows that the SRG of the class of $\epsilon$-strongly
                incrementally $(\mu,
        \gamma)$-dissipative systems is $\mathcal{D}_1 \cup \mathcal{D}_2$.}
        \label{fig:srg}
\end{figure}

\begin{proof}
        We begin by showing $\srg{\mathcal{S}_{\mu, \gamma}^\epsilon} \subseteq
        \mathcal{D}_1\cup\mathcal{D}_2$.
        Let  $H \in \mathcal{S}_{\mu, \gamma}^\epsilon$ and $u_1, u_2 \in L_2$ be arbitrary inputs.  Then, by assumption, for all $y_1 \in H(u_1)$, $y_2
        \in H(u_2)$, either \eqref{eq:small_gain} is true, or \eqref{eq:passivity_one} and \eqref{eq:passivity_two} are
        true, or all three inequalities are true.  Suppose first that \eqref{eq:small_gain} is true.
        Then $\norm{y_1 - y_2}/\norm{u_1- u_2} \leq \mu$, so $z_H(u_1, u_2) \subseteq
        \{z\, | \,z \in \C, |z| \leq \mu\} = \mathcal{D}_1$.

        We now treat the second case.  Suppose that \eqref{eq:passivity_one} and \eqref{eq:passivity_two} are true. 
        Note that, for $z \in z_H(u_1, u_2)$ corresponding to outputs $y_1$, $y_2$,
        \begin{IEEEeqnarray}{rCl}
                \Re(z) = \frac{\bra{u_1 - u_2}\ket{y_1 - y_2}}{\norm{u_1 - u_2}^2}.
        \end{IEEEeqnarray}
        It then follows from Equation~\eqref{eq:passivity_one} that $z_H(u_1, u_2) \subseteq
        \{z\, | \,z \in \C, \Re(z) \geq \epsilon\}$.
        Equation~\eqref{eq:passivity_two} gives $z_H(u_1, u_2) \subseteq
        \{z\, | \,z \in \C, |z| \leq \mu\}$, so $z_H(u_1, u_2) \subseteq \{z\, | \,z
        \in \C, \Re(z) \geq \epsilon\} \cap \{z\, | \,z \in \C, |z| \leq \mu\} =
        \mathcal{D}_2$.   Combining the two cases, we have $z_H(u_1, u_2) \subseteq
        \mathcal{D}_1\cup\mathcal{D}_2$.
        Since $u_1$ and $u_2$ were arbitrary, it follows that 
                $\srg{\mathcal{S}_{\mu, \gamma}^\epsilon} \subseteq \mathcal{D}_1\cup\mathcal{D}_2$.

        To show the opposite inclusion, let $z \in 
        \mathcal{D}_1\cup\mathcal{D}_2$ be arbitrary, and consider $A_z:L_2(\C) \to
        L_2(\C)$ defined by $A_z(w) = |z|e^{j\arg(z)}w$.  A straightforward calculation shows that $\srg{A_z}
= \{z, \bar{z}\}$.
        If we can show that $A_z \in \mathcal{S}_{\mu, \gamma}^\epsilon$, it follows that
                $\srg{\mathcal{S}_{\mu, \gamma}^\epsilon} \supseteq \mathcal{D}_1\cup\mathcal{D}_2$.
The fact that $A_z \in \mathcal{S}_{\mu, \gamma}^\epsilon$ is shown using the
following argument, which also proves SRG-fullness of $\mathcal{S}_{\mu,
\gamma}^\epsilon$.

        Take an arbitrary operator $H$ which
        satisfies $\srg{H} \subseteq \mathcal{D}_1\cup\mathcal{D}_2$.
        Take $z \in \srg{H}$, and let $u_1, u_2, y_1, y_2$ be any inputs and outputs
        that correspond to the point $z$.  If $z \in \mathcal{D}_1$, then $u_1, u_2,
        y_1, y_2$ obey \eqref{eq:small_gain}. If $z \in \mathcal{D}_2$, then $u_1, u_2,
        y_1, y_2$ obey \eqref{eq:passivity_one} and \eqref{eq:passivity_two}.  It follows that $H \in
        \mathcal{S}_{\mu, \gamma}^\epsilon$.
\end{proof}

We conclude this section with two examples of incrementally $(\mu,
\gamma)$-dissipative systems.
\begin{example}\label{ex:LTI-mg}
        We begin by revisiting the LTI transfer function of Example~\ref{ex:LTI}, $G
        = e^{-s}/(s+1)$.
        Its SRG is shown again in the left of Figure~\ref{fig:example_srgs}.  We can
        read directly from the SRG that this system is $(\mu_1,
        \gamma_1)$-dissipative, with $\mu_1 = 0.7581$ and $\gamma_1 = 1$.  The
        circles with centres at the origin and radii of $\mu_1$ and $\gamma_1$ are
        shown as dashed lines in Figure~\ref{fig:example_srgs}.  
\end{example}
\begin{figure}[ht]
        \centering
        \includegraphics{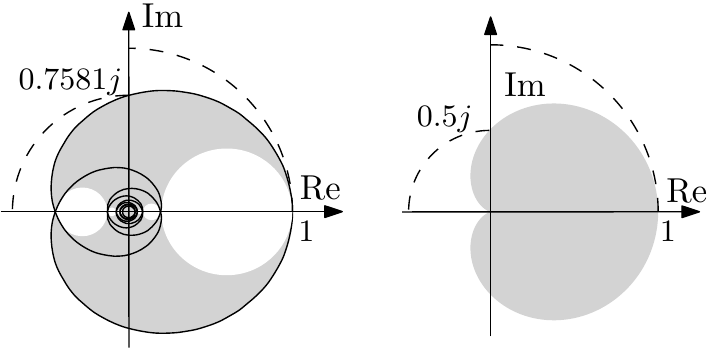}
        \caption{SRGs for Examples~\ref{ex:LTI-mg} (left) and~\ref{ex:sat-mg}
        (right), showing incremental $(\mu, \gamma)$-dissipativity, where $\mu$ and $\gamma$ are
        determined by the radii of the dashed circles.}%
        \label{fig:example_srgs}
\end{figure}
\begin{example}\label{ex:sat-mg}
        We now revisit the system of Example~\ref{ex:composition} -- the cascade of
        the transfer function $1/(s+1)$ with a unit saturation.  In
        Example~\ref{ex:composition}, a bounding approximation of the SRG of this
        system was obtained -- this is repeated on the right of
        Figure~\ref{fig:example_srgs}.  It can be read directly from the figure that
        this system is incrementally $(\mu_2, \gamma_2)$-dissipative, with $\mu_2 = 0.5$ and
        $\gamma_2 = 1$.  Again, the relevant circles are shown in the figure as
        dashed lines.  For any $\bar \mu >
        \mu_2$, it can be verified graphically that there exists an $\epsilon > 0$
        such that this system is $\epsilon$-strongly incrementally $(\bar \mu,
        \gamma_2)$-dissipative.
\end{example}

\section{A rolled-off passivity theorem}\label{sec:theorem}

\begin{theorem}\label{thm:small_passivity}
        Let $H_1, H_2: L_2 \to L_2$.  Suppose there exist $\mu_1, \mu_2,
        \gamma_1, \gamma_2 \geq 0$ and $\epsilon > 0$ such that $H_1$ is incrementally $(\mu_1,
        \gamma_1)$-dissipative
        and $H_2$ is $\epsilon$-strongly incrementally $(\mu_2, \gamma_2)$-dissipative.
        Then the feedback interconnection of $H_1$ and $H_2$ shown in
        Figure~\ref{fig:sym_fb} maps $L_2$ to $L_2$ and has finite incremental gain from $u$ to $y$ if
        \begin{IEEEeqnarray*}{c}
                \mu_1\mu_2 < 1, \qquad
                \mu_1\gamma_2 < 1, \qquad
                \mu_2\gamma_1 < 1.
        \end{IEEEeqnarray*}
\end{theorem}

\begin{figure}[ht]
        \centering
        \includegraphics{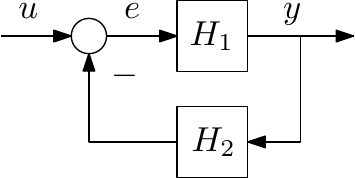}
        \caption{Negative feedback interconnection of $H_1$ and $H_2$.}%
        \label{fig:sym_fb}
\end{figure}

Note that setting $\mu_1 = \mu_2 = 0$  and letting $\gamma_1, \gamma_2 \to \infty$ recovers the incremental passivity theorem, and
setting $\gamma_1 < \mu_1$, $\gamma_2 < \mu_2$ recovers the incremental small gain theorem, for
operators on $L_2$ \cite{Desoer1975}.  The focus for the remainder of this paper will
be cases where neither the incremental passivity theorem nor the incremental small
gain theorem apply.  We do not make any assumptions about causality
of operators, and neither do we give any guarantees - causality must be treated
separately.

Theorem~\ref{thm:small_passivity} follows as a corollary of \cite[Thm.
2]{Chaffey2021c}.  In the remainder of this section, we give a direct proof.  
The proof proceeds by taking bounding approximations of the SRGs of
$H_1$ and $H_2$, and using them to construct a bounding approximation of the SRG of
the closed loop operator.  This approximation is a bounded region in the complex
plane, from which we conclude finite
incremental gain, via Proposition~\ref{prop:finite_gain}.

Given two arbitrary systems $H_1$ and $H_2$, it is not necessarily true that their
feedback interconnection is admissible, that is, a well-defined operator
on $L_2$.  If the systems satisfy the conditions of
Theorem~\ref{thm:small_passivity}, however, admissibility is guaranteed.  This is
proven using a homotopy argument similar to Megretski
and Rantzer \cite{Megretski1997}.  We scale the feedback operator by a gain $\tau \in
(0, 1]$, and show that the mapping from $\tau$ to the closed loop incremental
gain is continuous.  This shows the finite incremental gain of $H_1$ is preserved as
the feedback is introduced, and in particular, the closed loop system continues to
map $L_2$ into $L_2$.  Note that the
standard form of the incremental passivity theorem \cite[Thm. 30, p.
184]{Desoer1975} requires admissibility as an assumption.  The extra strength
of Theorem~\ref{thm:small_passivity} comes, loosely speaking, from the additional
assumption that both operators $H_1$ and $H_2$ have finite incremental gain.

\begin{proof}[Proof of Theorem~\ref{thm:small_passivity}]
        Let $\mu_1, \mu_2, \gamma_1, \gamma_2$ and $\epsilon$ satisfy the conditions of the theorem.
        Let $\mathcal{S}_{\mu_1, \gamma_1}$ be the class of operators which are
        incrementally $(\mu_1,
        \gamma_1)$-dissipative, and $\mathcal{S}_{\mu_2, \gamma_2}^\epsilon$ be the class of operators which are
        $\epsilon$-strongly incrementally $(\mu_2, \gamma_2)$-dissipative.

        The class of operators formed by the negative feedback interconnection of
        $H_1 \in \mathcal{S}_{\mu_1, \gamma_1}$ and $\tau H_2 \in \tau \mathcal{S}_{\mu_2,
        \gamma_2}^\epsilon$ (Figure~\ref{fig:sym_fb}), is given by
        \begin{IEEEeqnarray}{rCl}
                (\mathcal{S}_{\mu_1, \gamma_1}^{-1} + \tau \mathcal{S}_{\mu_2,
                \gamma_2}^\epsilon)^{-1}.\label{eq:class}
        \end{IEEEeqnarray}
        Let $\tau \in (0, 1]$.
        We now plot the SRGs of $\mathcal{S}_{\mu_1, \gamma_1}^{-1}$ and
        $-\tau \mathcal{S}_{\mu_2, \gamma_2}^\epsilon$, using Lemma~\ref{lem:srg} and
        Propositions~\ref{prop:gain}, \ref{prop:inversion} and \ref{prop:summation}.
        \begin{center}
                \includegraphics{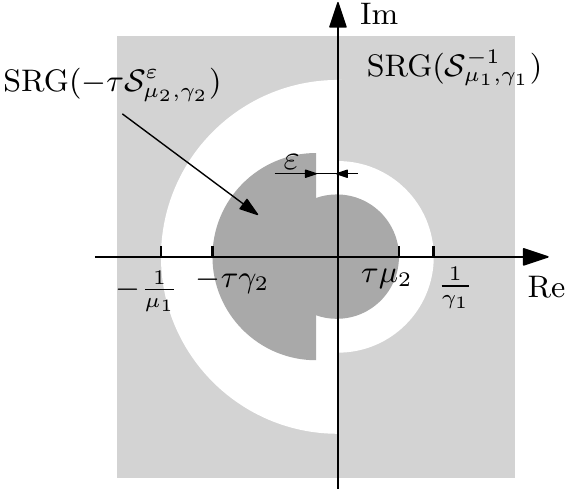}
        \end{center}
        The shortest distance between the two SRGs, that is, $\min \{|z_1 - z_2|\}$
        over all $z_1 \in \srg{\mathcal{S}_{\mu_1, \gamma_1}^{-1}}$ and all
$z_2 \in \srg{-\tau\mathcal{S}_{\mu_2, \gamma_2}^\epsilon}$, is given by
        \begin{IEEEeqnarray*}{rCl}
                r_\tau \coloneqq \min\left\{\epsilon, \frac{1}{\mu_1} - \tau\mu_2,
                \frac{1}{\mu_1} - \tau\gamma_2, \frac{1}{\gamma_1} - \tau\mu_2\right\}.
        \end{IEEEeqnarray*}
        This is determined from the figure above, allowing for the fact that one or
        both SRGs may have $\mu_j > \gamma_j$.
        By the assumption of the theorem, all of these values are positive for $\tau
        \in (0, 1]$. Applying
        Propositions~\ref{prop:gain} (with $\alpha = -1$) and~\ref{prop:summation}, it follows that $\srg{\mathcal{S}_{\mu_1, \gamma_1}^{-1} +
        \tau\mathcal{S}_{\mu_2, \gamma_2}^\epsilon} = \srg{\mathcal{S}_{\mu_1,
        \gamma_1}^{-1}} -
        \srg{-\tau\mathcal{S}_{\mu_2, \gamma_2}^\epsilon}$ is bounded away from zero by a distance $r_\tau$. Applying
        Propositions~\ref{prop:finite_gain}
        and \ref{prop:inversion} allows us to conclude a finite incremental gain bound of
        $1/r_\tau$ for the class of operators \eqref{eq:class}, as illustrated below.
        \begin{center}
                \includegraphics{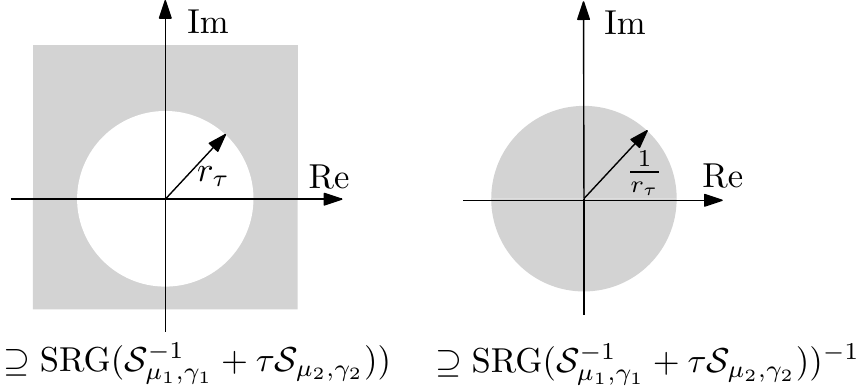}
        \end{center}

        Let $\epsilon \geq 0$ be less than $r_\tau$. Then there is a
        $\delta(\epsilon) \geq 0$, with $\delta(0) = 0$, such that $\tau +
        \delta(\epsilon) \leq 1$ and
        the distance between $\srg{\mathcal{S}_{\mu_1, \gamma_1}^{-1}}$ and 
        $\srg{-\tau\mathcal{S}_{\mu_2, \gamma_2}^\epsilon}$ decreases to $r_{\tau +
        \delta(\epsilon)} = r_\tau - \epsilon$.  The closed loop incremental gain is
        then bounded by $1/(r_\tau - \epsilon)$.  This shows that the scaling factor
        $\tau$ maps continuously to the closed loop gain, so the finite incremental
        gain of the forward loop is preserved as feedback is introduced, and the
        closed loop continues to map $L_2 \to L_2$.  The theorem then follows by
        taking $\tau = 1$.
\end{proof}

\section{Feedback example}\label{sec:example}

\begin{example}\label{ex:feedback}

\begin{figure}[hb]
        \centering
        \includegraphics{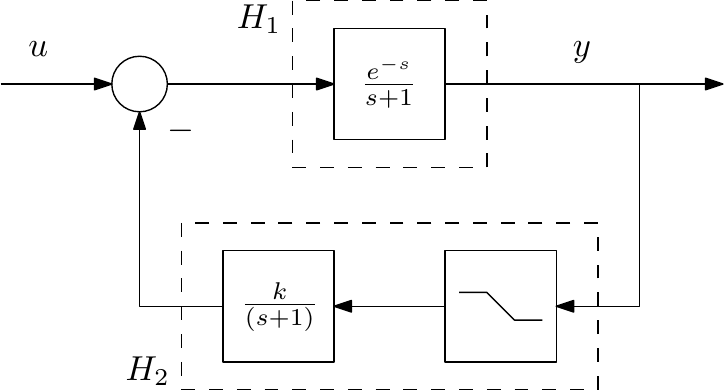}
        \caption{Feedback system studied in Example~\ref{ex:feedback}.}%
        \label{fig:example_system}
\end{figure}

Consider the feedback system shown in Figure~\ref{fig:example_system}.  The forward
path $H_1$ consists of a delayed first order lag. The feedback path $H_2$ consists of a
unit saturation cascaded with a first order lag and static gain $k$.  

The SRG of
$H_1$ is derived in Example~\ref{ex:LTI}, and $H_1$ is shown to be incrementally $(\mu_1,
\gamma_1)$-dissipative, with $\mu_1 = 0.7581$ and $\gamma_1 = 1$, in
Example~\ref{ex:LTI-mg}.  

The SRG of $H_2$ is obtained by applying Proposition~\ref{prop:gain} to the SRG obtained 
in Example~\ref{ex:composition}, to give the cardioid 
\begin{IEEEeqnarray*}{rCl}
        \{kre^{j\phi}\; |\; r \leq \frac{1}{2}(1 + \cos(\phi))\},
\end{IEEEeqnarray*}
that is, the SRG of Example~\ref{ex:composition} inflated by $k$.  It then follows
from Example~\ref{ex:sat-mg} that $H_2$ is $\epsilon$-strongly incrementally $(\mu_2,
\gamma_2)$-dissipative, with $\mu_2 > k/2$,
        $\gamma_2 = k$ and $\epsilon > 0$.

Note that neither $H_1$ nor $H_2$ is incrementally positive (their SRGs are not 
contained in the right half plane), nor do they obey the incremental small gain 
condition for $k > 1$ (the product of the maximum moduli of their SRGs exceeds 1).
However, it follows from Theorem~\ref{thm:small_passivity} that the feedback system
has finite incremental $L_2$ gain for all $0 < k < 1.3191$.
\end{example}

\section{Conclusions}

Theorem~\ref{thm:small_passivity} guarantees finite incremental gain of the negative 
feedback interconnection of two systems, where, for every pair of input/output
trajectories, the 
systems satisfy either an incremental small gain condition or an incremental 
passivity condition.  This property, which we call incremental $(\mu, \gamma)$-dissipativity, 
captures systems which have either small gain or small phase shift (or both),
and includes systems with low-pass dynamics whose passivity is destroyed at high
frequencies by effects such as saturation and delay. 

A primary advantage of incremental $(\mu, \gamma)$-dissipativity is that it
can be verified graphically from the SRG of a system, as shown in
Lemma~\ref{lem:srg}.  This makes the property both intuitive, and simple to verify.

\section*{Acknowledgements}
The research leading to these results has received funding from the European
Research Council under the Advanced ERC Grant Agreement Switchlet n.~670645.
The author gratefully acknowledges many insightful discussions with Fulvio Forni and
Rodolphe Sepulchre, and useful suggestions on the manuscript from Alberto Padoan and
the anonymous reviewers.

\bibliographystyle{elsarticle-num-names}
\bibliography{srg}
\end{document}